\documentclass[sn-mathphys-num]{sn-jnl}
\usepackage{graphicx}%
\usepackage{multirow}%
\usepackage{amsmath,amssymb,amsfonts}%
\usepackage{amsthm}%
\usepackage{mathrsfs}%
\usepackage[title]{appendix}%
\usepackage{xcolor}%
\usepackage{textcomp}%
\usepackage{manyfoot}%
\usepackage{booktabs}%
\usepackage{algorithm}%
\usepackage{algorithmicx}%
\usepackage{algpseudocode}%
\usepackage{listings}%

\theoremstyle{thmstyleone}%
\newtheorem{theorem}{Theorem}[section]
\newtheorem{proposition}[theorem]{Proposition}
\newtheorem{lemma}[theorem]{Lemma}
\newtheorem{corollary}[theorem]{Corollary}
\newtheorem{example}[theorem]{Example}
\newtheorem{definition}[theorem]{Definition}

\begin{document}

\title[$\Sigma$-dual-Rickart Modules]{$\Sigma$-dual-Rickart Modules}


\author[1]{\fnm{Shiv} \sur{Kumar}}\email{shivkumar.rs.mat17@itbhu.ac.in}

\author[2]{\fnm{Ashok Ji} \sur{Gupta}}\email{agupta.apm@itbhu.ac.in}


\affil[1]{\orgdiv{Department of Mathematical Sciences}, \orgname{IIT(BHU)}, \orgaddress{ \city{Varanasi}, \postcode{221005}, \state{Uttar Pradesh}, \country{India}}}

\affil[2]{\orgdiv{Department of Mathematical Sciences}, \orgname{IIT(BHU)}, \orgaddress{ \city{Varanasi}, \postcode{221005}, \state{Uttar Pradesh}, \country{India}}}



\abstract{In this paper, we introduce and study the concept of {\it $\Sigma$-dual-Rickart} modules which is the dual case of  $\Sigma$-Rickart modules. An $A$-module $X$ is said to be a {\it $\Sigma$-dual-Rickart} if every direct sum of copies of $X$ is a dual-Rickart module.  We show when a $\Sigma$-Rickart module is a {\it $\Sigma$-dual-Rickart} module and vice-versa. Also, we characterize semisimple Artinian rings, regular rings, semi-hereditary rings, cohereditary modules and $FP$-injective  modules in terms of {\it $\Sigma$-dual-Rickart} modules. Further, we study the endomorphism ring of a {\it $\Sigma$-dual-Rickart} module.}




\maketitle

	\section{Introduction}\label{sec1}
In this article, all rings are associative with unity and all modules are unital right $A$-modules, unless otherwise specified. We recall that a ring $A$ is called $Baer$ \cite{IK} if the right annihilator of every right ideal in $A$ is generated by an idempotent element of $A$ as a right ideal of $A$, while a ring $A$ is called $Rickart$ \cite{Maeda} if the right annihilator of every element of $A$ is generated by an idempotent element of $A$ as a right ideal of $A$. The structural characteristics of $Baer$ and $Rickart$ rings exhibit a close connection with $C^*$-algebras and von Neumann algebras, rooted in the principles of functional analysis. Also, the theory of $Rickart$ rings plays a significant role in the structure theory of rings. Moreover, the theory of $Rickart$ rings plays a crucial role in the study of ring structures. Therefore, the exploration of the module theoretical concepts related to $Baer$ and $Rickart$ rings has emerged as an appealing subject for researchers.
In \cite{LRRR} and \cite{LRRD}, Lee et al. introduced the notion of $Rickart$ modules and $dual$-$Rickart$ (in short we will call $d$-$Rickart$) modules, respectively. They call a module $X$ $Rickart$ ($d$-$Rickart$) if $\operatorname{Ker}(\phi)$ ($\operatorname{Im}(\phi)$) is a direct summand of $X$ for every $\phi \in \operatorname{End}_A(X)$. Over the recent years, numerous researchers have contributed to the advancement of theory of $Rickart$ and $d$-$Rickart$ modules. (see \cite{AH}, \cite{Ext}, \cite{LRRR}, \cite{LRRD}, \cite{LM1}). It is observed that the direct sum of Rickart modules may not necessarily be a Rickart module (see \cite{LRRS}). In \cite{LM1}, Lee and Barcenas introduced $\Sigma$-$Rickart$ modules and they called a module $X$, $\Sigma$-$Rickart$ if every direct sum of copies of $X$ is a $Rickart$ module.

Motivated by the notion of $\Sigma$-Rickart modules, we introduce the notion of {\it $\Sigma$-dual-Rickart} (or in short {\it $\Sigma$-d-Rickart}) modules in Section 2. We call a module $X$ $\Sigma$-$d$-$Rickart$, if the direct sum of arbitrary many copies of $X$ is a $d$-$Rickart$ module. We prove that {\it $\Sigma$-d-Rickart} modules are closed under direct summands and direct sums. Recall from \cite{MS} that a module $X$ is $hereditary$ $(cohereditary)$ if every submodule (factor module) of $X$ is $projective$ $(injective)$. We show that every cohereditary module over a Noetherian ring is $\Sigma$-$d$-$Rickart$ module (see, Proposition \ref{S2.6}).

In section 3, we provide some examples which show that a $\Sigma$-$d$-$Rickart$ module need not be a $\Sigma$-$Rickart$ module and vice-versa (see Example \ref{SDR1.1}). An $A$-module $X$ is said to be a $\Sigma$-$\operatorname{C2}$ ($\Sigma$-$\operatorname{D2}$) module if every direct sum of copies of $X$ is a $\operatorname{C2}$ ($\operatorname{D2}$)-module. We show that every $\Sigma$-$Rickart$ ($\Sigma$-$d$-$Rickart$) module $X$ is a $\Sigma$-$d$-Rickart ($\Sigma$-Rickart) module if the module $X$ is $\Sigma$-$\operatorname{C2}$ ($\Sigma$-$\operatorname{D2}$) (see Proposition \ref{SDR1.2} (Proposition \ref{SDR1.4})).  We also prove that every $A$-module is {\it $\Sigma$-d-Rickart} if and only if every $A$-module is $\Sigma$-Rickart if and only if the ring $A$ is semisimple Artinian (see Theorem \ref{SDR1.5}).

In section 4, we study the endomorphism ring of a {\it {\it $\Sigma$-d-Rickart}} module. We support this notion by presenting an example that demonstrates the endomorphism ring of a {\it $\Sigma$-d-Rickart} module, which may not be hereditary (see Example \ref{S3.1}), while it is left semi-hereditary (see Proposition \ref{S3.2}). Further, we prove that the endomorphism ring of a finitely generated {\it $\Sigma$-d-Rickart} module is hereditary (see Proposition \ref{S3.4}). Examples and counterexamples are given, which delineate our results.
\section{Preliminaries}
{\bf Notations} $\subseteq$, $\leq$, $\leq ^{\oplus}$ are used to denote a subset, a submodule and a direct summand respectively. If $X$ is an $A$-module and $\Lambda$ is a non-empty index set, then $\bigoplus_{\Lambda} X$ or $X^{(\Lambda)}$ and $\Pi_{\Lambda}X$ or $X^{\Lambda}$, denote the direct sum and the direct product of $\Lambda$-copies of $X$, respectively. $CFM$ denotes the column finite matrix ring, $A$-Mod denotes the category of all left $A$-modules, Mod-$A$ denotes the category of all right $A$-modules and $E(X)$ denotes the injective hull of an $A$-module $X$.\\
\begin{definition}
	A module $X$ is said to have the $\operatorname{C2}$-condition if any submodule of $X$ which is  isomorphic to a direct summand of $X$ is also a direct summand of $X$. A module with $\operatorname{C2}$-condition is known as $\operatorname{C2}$-module \cite{BJM}.
\end{definition}
~\\
\begin{definition}
	A module $X$ is said to have the $\operatorname{D2}$-condition if for any submodule $N\leq X$, $X/N$ is isomorphic to a direct summand of $X$, then $N$ is also a direct summand of $X$. A module with $\operatorname{D2}$-condition is known as $\operatorname{D2}$-module \cite{BJM}.
\end{definition}
~\\
\begin{definition}
	A module $X$ is called a $\operatorname{C3}$-module if whenever $Y_1$ and $Y_2$ are submodules of $X$ such that $Y_1\leq ^{\oplus} X$, $Y_2\leq ^{\oplus} X$ and $Y_1\cap Y_2=0$, then $Y_1\oplus Y_2\leq ^{\oplus} X$ \cite{C3}.
\end{definition}
~\\
\begin{definition}
	A module $X$ is called a $\operatorname{D3}$-module if whenever $N_1$ and $N_2$ are submodules of $X$ such that $N_1\leq ^{\oplus} X$, $N_2\leq ^{\oplus} X$ and $N_1+N_2=X$ then $N_1\cap N_2\leq ^{\oplus} X$ \cite{D3}.
\end{definition}
~\\
\begin{definition}
	(16.9, \cite{RWF}) Let $X$ and $N$ be $A$-modules. The module $N$ is called a $weakly$ $X$-$injective$ module, if for every finitely generated submodule $L$ of $X^{(\mathbb{N})}$ with a monomorphism $\phi : L\rightarrow X^{(\mathbb{N})}$ and for every homomorphism $\psi : L\rightarrow N$ there exists a homomorphism $\eta : X^{(\mathbb{N})}\rightarrow N$ such that $\psi=\eta \phi$. If $X=R$, then weakly $A$-injective modules are also called $FP$-$injective$.
\end{definition}
~\\
\begin{definition}
	A module $X$ is called $endoregular$ if the endomorphism ring of $X$ is a von Neumann regular ring \cite{RRER}.
\end{definition}
~\\
\begin{lemma}\label{LD}
	(Theorem 2.19, \cite{LRRD}). Let $X_1$ and $X_2$ be $A$-modules. Then $X_1$ is $X_2$-d-Rickart if and only if for any direct summand $N_1 \leq ^{\oplus} X_1$ and any submodule $N_2 \leq X_2$, $N_1$ is $N_2$-d-Rickart.
\end{lemma}
~\\
\begin{lemma}\label{LM1.1}
	Let $X$ and $N$ be $A$-modules such that $\psi :X\rightarrow N$ be a monomorphism and $Y\leq X$. If $\psi(Y)$ is a direct summand of $N$, then $Y$ is a direct summand of $X$.
	\begin{proof}
		Suppose that $N=\psi(Y){\oplus} Z$ for some $Z\leq N$. Since $\psi$ is a monomorphism, so $X=Y + \psi ^{-1}[Z]$, where $\psi ^{-1}[Z]$ is the set of inverse images of $Z$ under $\psi$. If $m\in Y\cap \psi ^{-1}[Z]$, then there exists $l\in Z$ such that $\psi(m)=l$. Since $m\in Y$, $y\in \psi (Y)$, therefore $l\in \psi (Y)\cap Z=0$, which implies $m=0$. Hence $X=Y\oplus \psi
		^{-1}[Z]$.
	\end{proof}
\end{lemma}
~\\
\begin{lemma}\label{LI}
	(Corollary, \cite{BL}). Let $A$ be a ring that contains an infinite direct product $\Pi_{i\in I} A_i$, where $A_i$ is a ring with identity $e_i$ for $i\in I$. Then $A$ is not a (right) hereditary ring.
\end{lemma}
\begin{lemma}\label{FP1.4}
	(47.7(2), \cite{RWF}). Let $X$ be a right $A$-module with $S=\operatorname{End}_A(X)$. Then ${_SX}$ is  $FP$-injective $S$-module if and only if for every homomorphism\\ $\phi:  {X^{(n_1)}}\rightarrow  {X^{(n_2)}}$ with $n_1,n_2\in \mathbb{N}$, $\operatorname{Coker}(\phi)$ is a $X$-cogenerated module.
\end{lemma}
\section{$\Sigma$-dual-Rickart modules}

\begin{definition}\label{S2.1}
	An $A$-module $X$ is said to be {\it $\Sigma$-d-Rickart} (in short {\it $\Sigma$-d-Rickart} ) if every direct sum of copies of $X$ is d-Rickart. Equivalently, a module $X$ is {\it $\Sigma$-d-Rickart} if $X^{(\Lambda)}$ is d-Rickart for every non-empty arbitrary index set $\Lambda$. A ring $A$ is called a right (left) {\it $\Sigma$-d-Rickart} if $A_A$ $(_AA)$ is a {\it $\Sigma$-d-Rickart} $A$-module and a ring $A$ is called {\it $\Sigma$-d-Rickart} if $A$ is a right and left {\it $\Sigma$-d-Rickart} ring.
\end{definition}
~\\
\begin{example}\label{S2.2}
	\begin{enumerate}
		\item[(i)] The $\mathbb{Z}$-module $\mathbb{Z}_{p^{\infty}}$ is a {\it $\Sigma$-d-Rickart}.
		\item[(ii)] Every injective $A$-module over the hereditary Noetherian ring $A$ is {\it $\Sigma$-d-Rickart} (see Theorem \ref{T2.17}). 
		\item[(iii)] If the ring $A$ is semisimple Artinian, then every $A$-module is {\it $\Sigma$-d-Rickart} (see Theorem \ref{SDR1.5}).
	\end{enumerate}
~\\
\end{example}
It is easy to see that if $X$ is a {\it $\Sigma$-d-Rickart} module, then $X$ is also a d-Rickart module. Now, we provide an example which shows that a d-Rickart module need not be {\it $\Sigma$-d-Rickart} module. For instant, let $A=\Pi_{n=1}^{\infty} F_n$ where $F_n=\mathbb{Z}_p$ for each $n$ and $p$ is any prime. Clearly, $A$ is a self-injective von Neumann regular ring. By (Proposition 2.27, \cite{LRRD}), $X=A$ is a d-Rickart $A$-module, while the module $X^{(A)}$ is not a d-Rickart $A$-module (see Example 3.9, \cite{LRRD}). Hence $X$ can not be a {\it $\Sigma$-d-Rickart} $A$-module.\\~\\
Now, we prove when a d-Rickart module is {\it $\Sigma$-d-Rickart}.\\

\begin{proposition}
	Let $X$ be an $A$-module such that $X$ is fully invariant in $X^{(\Lambda)}$ for every arbitrary index set $\Lambda$. Then
	\begin{enumerate}
		\item[(i)] $X$ is a d-Rickart module if and only if $X$ is {\it $\Sigma$-d-Rickart}.
		\item[(ii)] $X$ is a Rickart module if and only if $X$ is $\Sigma$-Rickart.
		
	\end{enumerate} 
	\begin{proof}
		${\bf (i)}$ Let $X$ be a d-Rickart module and fully invariant in $X^{(\Lambda)}$. So by (Proposition 5.14, \cite{LRRD}), $X^{(\Lambda)}$ is a d-Rickart module. Hence, $X$ is a {\it $\Sigma$-d-Rickart} module. The converse is clear from the definition of {\it $\Sigma$-d-Rickart} module.\\
		${\bf (ii)}$ The proof of this part is dual of part $(i)$ and it is consequence of  (Proposition 2.34, \cite{LRRS}).
	\end{proof}
\end{proposition}

\begin{lemma}\label{S2.3}
	\begin{enumerate}
		\item[(i)] Every direct summand of a {\it $\Sigma$-d-Rickart} module is {\it $\Sigma$-d-Rickart}.
		\item[(ii)] Every direct sum of copies of a {\it $\Sigma$-d-Rickart} module is {\it $\Sigma$-d-Rickart}.
		\item[(iii)] If $X$ is a {\it $\Sigma$-d-Rickart} module, then for any index sets $\Lambda _1$ and $\Lambda _2$, $X^{(\Lambda _1)}$ is $X^{(\Lambda _2)}$ d-Rickart.
	\end{enumerate}
	
	\begin{proof}
		${\bf (i)}$ Let $X$ be a {\it $\Sigma$-d-Rickart} module and $N\leq ^{\oplus}X$. Then for any index set $\Lambda$, $N^{(\Lambda)}$ is also a direct summand of $X^{(\Lambda)}$. Since $X^{(\Lambda)}$ is a d-Rickart module, $N^{(\Lambda)}$ is a d-Rickart module. Hence, $N$ is a {\it $\Sigma$-d-Rickart} module.\\
		${\bf (ii)}$ 	Let $X$ be a {\it $\Sigma$-d-Rickart} module. Then $X^{(\Lambda_1)}$ is a d-Rickart module for every non-empty arbitrary index set $\Lambda_{1}$. Therefore, by the definition of {\it $\Sigma$-d-Rickart} module, $(X^{(\Lambda_1)})^{(\Lambda_2)}=X^{(\Lambda_1 \times \Lambda_2)}$ is a d-Rickart module for every index set $\Lambda_2$. Hence, $X^{(\Lambda_1)}$ is a {\it $\Sigma$-d-Rickart} module.\\
		${\bf (iii)}$ 	Let $X$ be a {\it $\Sigma$-d-Rickart} module. Then $X^{(\Lambda _1\times \Lambda _2)}$ is also a d-Rickart module. Let $\varphi : X^{(\Lambda _1)}\rightarrow X^{(\Lambda _2)}$ be a homomorphism. It is clear that $X^{(\Lambda_1)}$ and $X^{(\Lambda_2)}$ are direct summands of $X^{(\Lambda _1\times \Lambda _2)}$. Therefore, by Lemma \ref{LD} $\operatorname{Im}(\varphi)\leq ^{\oplus}X^{(\Lambda_2)}$. Hence, $X^{(\Lambda_1)}$ is $X^{(\Lambda_2)}$-d-Rickart.
	\end{proof}
\end{lemma}

A right $A$-module $X$ is called cohereditary if every factor module of $X$ is an injective module \cite{RWF}.\\

In the following proposition, we prove that every cohereditary module over the Noetherian ring is {\it $\Sigma$-d-Rickart}.\\

\begin{proposition}\label{S2.6}
	Every cohereditary module $X$ over the Noetherian ring $A$ is a {\it $\Sigma$-d-Rickart} $A$-module.
	\begin{proof}
		Let $X$ be a cohereditary module and $A$ be a Noetherian ring. By (Remark (i), \cite{MS}), $X^{(\Lambda)}$ is a cohereditary module for any arbitrary index set $\Lambda$. So, for any $\psi \in \operatorname{End}_A(X^{(\Lambda)})$, $X^{(\Lambda)}/\operatorname{Ker}(\psi)\cong \operatorname{Im}(\psi)$ is an injective $A$-module. Therefore, the exact sequence $0\rightarrow \operatorname{Im}(\psi)\rightarrow X^{(\Lambda)}\rightarrow \operatorname{Coker}(\psi)\rightarrow 0$ splits, which implies that $\operatorname{Im}(\psi)$ is a direct summand of $X^{(\Lambda)}$. Hence, $X$ is a {\it $\Sigma$-d-Rickart} module.
	\end{proof} 	
\end{proposition}

\begin{proposition}
	If $X$ is an $A$-module such that $X^{(\Lambda)}$ is an endoregular module for every arbitrary index set $\Lambda$, then $X$ is a {\it $\Sigma$-d-Rickart} module. 
	\begin{proof}
		The proof is clear from (Proposition 2.3, \cite{RRER}).
	\end{proof}
\end{proposition}
~\\
Let $X$ and $N$ be $A$-modules. Then, $N$ is called an $X$-$cogenerated$ module if there exists a monomorphism from $N$ to $X^{\Lambda}$ for every non-empty index set $\Lambda$ and $N$ is called a $finitely$ $X$-$cogenerated$ module if there exists a  monomorphism from $N$ to $X^{\mathcal{I}}$ for every finite subset $\mathcal{I}$ of $\Lambda$ (see \cite{Anderson} and \cite{RWF}). Now, we generalize the concept of $X$-cogenerated modules as strongly $X$-cogenerated modules.\\

\begin{definition}\label{S2.7}
	Let $\mathcal{M}$ be a non-empty class of right $A$-modules. We call a module $N$ strongly cogenerated by $\mathcal{M}$ if there is a monomorphism $\sigma : N\rightarrow \bigoplus_{\lambda \in \Lambda} X_{\lambda}$, where $X_{\lambda}\in \mathcal{M}$ and $\Lambda$ is a non-empty arbitrary index set. An $A$-module $N$ is said to be strongly cogenerated by a module $X$ (or strongly $X$-cogenerated), if there exists a monomorphism $\sigma : N\rightarrow$ $X^{(\Lambda)}$ for every non-empty arbitrary index set $\Lambda$.
\end{definition}
~\\~
\begin{example}\label{S2.8}
	\begin{enumerate}
		\item[(i)] Every finitely $X$-cogenerated modules are strongly $X$-cogenerated.
		\item[(ii)]  Every strongly $X$-cogenerated modules are $X$-cogenerated, while the converse need not be true. Since every torsion-free abelian group is cogenerated by $\mathbb{Q}$ (Example 8.3, \cite{Anderson}), therefore $ \mathbb{Z}^{\mathbb{N}}$ is cogenerated by $\mathbb{Q}$ but it is not strongly cogenerated by $\mathbb{Q}$. In fact, $\mathbb{Z}^{\mathbb{N}}$ is not embedded in $\mathbb{Q}^{(\mathbb{N})}$.
	\end{enumerate}
\end{example}
~\\
\begin{lemma}\label{S2.9}
	The direct sum of any two strongly $X$-cogenerated modules is a strongly $X$-cogenerated.
	\begin{proof}
		Let $X_1$ and $X_2$ be two strongly $X$-cogenerated modules. Then for some non-empty index sets $\Lambda _1$ and $\Lambda _2$,  $\psi _1 : X_1\rightarrow X^{(\Lambda _1)}$ and $\psi _2 :X_2 \rightarrow X^{(\Lambda _2)} $ are monomorphisms. Therefore, by (9.2, \cite{RWF}) $\psi_1\oplus \psi_2 : X_1 \oplus X_2 \rightarrow X^{(\Lambda _1)}\oplus X^{(\Lambda _2)}$ is a monomorphism. Hence, $X_1\oplus X_2$ is also strongly $X$-cogenerated.
	\end{proof}
\end{lemma}

For an $A$-module $X$, $Add(X)$ denotes the class of all $A$-modules which are direct summands of direct sums of copies of $X$ \cite{RWT}.\\

\begin{proposition}\label{S2.11}
	Let $X$ be a {\it $\Sigma$-d-Rickart} module and $U\in Add(X)$. Then, the sum of two strongly $X$-cogenerated submodules of $U$ is strongly $X$-cogenerated.
	\begin{proof}
		Let $U_1$ and $U_2$ be two strongly $X$-cogenerated submodules of $U$ and $U\in Add(X)$. Consider the short exact sequence:
		$$ 0\rightarrow U_1 \cap U_2\xrightarrow{f} U_1 \oplus U_2\xrightarrow{g} U_1 +U_2\rightarrow 0$$
		where for any $a\in U_1 \cap U_2$, $f(a)=(a,a)$ and $g(u_1, u_2)=u_1+u_2$ where $u_1\in U_1$, $u_2\in U_2$. Since $U_1$ and $U_2$ are strongly $X$-cogenerated, therefore by Lemma \ref{S2.9} $U_1 \oplus U_2$ is strongly $X$-cogenerated. So there exists a monomorphism $\phi : U_1\oplus U_2 \rightarrow X^{(\Lambda)}$ for some index set $\Lambda$. Since $X$ is a {\it $\Sigma$-d-Rickart} module, so from Lemma \ref{S2.3}(iii) $U$ is a $X^{(\Lambda)}$-d-Rickart because $U\in Add(X)$. Therefore $ U_1 \cap U_2$ is also $X^{(\Lambda)}$-d-Rickart. So by Lemma \ref{LD}, $\operatorname{Im}(\phi f)$ is a direct summand of $X^{(\Lambda)}$. Thus from Lemma \ref{LM1.1}, $\operatorname{Im}(f)\leq ^{\oplus} U_1\oplus U_2$ because $\operatorname{Im}(\phi f)=\phi (\operatorname{Im}(f))$, where $\phi$ is a monomorphism. Therefore, $(U_1\oplus U_2)/\operatorname{Im}(f)$ is strongly $X$-cogenerated. Since $(U_1\oplus U_2)/\operatorname{Im}(f)\cong U_1+U_2$, hence $U_1+U_2$ is strongly $X$-cogenerated.
	\end{proof}
\end{proposition}

In the following proposition we show when every $A$-module is {\it $\Sigma$-d-Rickart}.\\

\begin{proposition}\label{S2.15}
	An $A$-module $X$ is {\it $\Sigma$-d-Rickart} if and only if every element in $Add(M)$ is {\it $\Sigma$-d-Rickart}.
	\begin{proof}
		Let $X$ be a {\it $\Sigma$-d-Rickart} module and $N\in Add(X)$ be arbitrary. Then, $X^{(\Lambda)}$ is a d-Rickart module and $N\leq ^{\oplus} X^{(\Lambda)}$ for an index set $\Lambda$. Hence, from Lemma \ref{S2.3} $N$ is a {\it $\Sigma$-d-Rickart} module. The converse follows from the definition of {\it $\Sigma$-d-Rickart} module.	
	\end{proof}
\end{proposition}

The following theorem illustrates when every injective module is {\it $\Sigma$-d-Rickart}. The proof of the part $(i)\Rightarrow (ii)$ follows from (Theorem 2.29, \cite{LRRD}) but we include its proof for reader's convenience.\\

\begin{theorem}\label{T2.17}
	Let $A$ be a Noetherian ring, then the following conditions are equivalent:
	\begin{enumerate}
		\item[(i)] Every injective $A$-module is a $\Sigma$-dual
		Rickart module;
		\item[(ii)] $A$ is a right hereditary ring.
	\end{enumerate}
	\begin{proof}	
		$(i)\Rightarrow (ii)$ Let $X$ be an injective $A$-module and $N$ be a submodule of $X$. Clearly, $X$ and $E(X/N)$ both are injective $A$-modules, so $X\oplus E(X/N)$ is also an injective $A$-module. By hypothesis $X\oplus E(X/N)$ is a {\it $\Sigma$-d-Rickart} module. So $X\oplus E(X/N)$ is a d-Rickart module. Thus from Lemma \ref{LD}, $X$ is $E(X/N)$-d-Rickart. Now consider a map $\psi : X\rightarrow E(X/N)$ such that $\psi (\zeta)=\zeta +N$ for every $\zeta \in X$. Then, $\operatorname{Im}(\psi)=X/N$ is a direct summand of $E(X/N)$. So, $X/N$ is an injective module. Hence $A$ is a right hereditary ring.\\
		$(ii)\Rightarrow (i)$ Let $A$ be a Noetherian ring and $X$ be an injective module. Then $X^{(\Lambda)}$ is also an injective $A$-module. Now, suppose that $\phi \in \operatorname{End}_A(X^{(\Lambda)})$. Since by assumption $A$ is a hereditary ring, $\operatorname{Im}(\phi) \cong X^{(\Lambda)} /\operatorname{Ker}(\phi)$ is an injective module. Therefore, $\operatorname{Im}(\phi)$ is a direct summand of $X^{(\Lambda)}$. Hence $X$ is a {\it $\Sigma$-d-Rickart} module.
	\end{proof}
\end{theorem}


\section{$\Sigma$-Rickart modules versus {\it $\Sigma$-d-Rickart} modules}

In this section, we find connections between the class of $\Sigma$-Rickart modules and the class of {\it $\Sigma$-d-Rickart} modules. Further, we show that  when a {\it $\Sigma$-d-Rickart} module is $\Sigma$-Rickart and vice-versa.\\~\\
The following examples justify that $\Sigma$-Rickart modules and {\it $\Sigma$-d-Rickart} modules are two different types of structure.\\

\begin{example}\label{SDR1.1}
	\begin{enumerate}
		\item[(i)] It is clear from (Theorem 2.29, \cite{LRRD}) that for the right hereditary ring $A$, every injective right $A$-module is d-Rickart. Therefore, $\mathbb{Z}_{p^{\infty}}^{(\Lambda)}$ is a d-Rickart $\mathbb{Z}$-module for every arbitrary index set $\Lambda$. Thus $\mathbb{Z}_{p^{\infty}}$ is a {\it $\Sigma$-d-Rickart} $\mathbb{Z}$-module, while $\mathbb{Z}_{p^{\infty}}$ is not a $\Sigma$-Rickart module. In fact, $\mathbb{Z}_{p^{\infty}}$ is not a Rickart module (see Example 2.17 \cite{LRRR}).
		\item[(ii)] Since from (Theorem 2.26, \cite{LRRR}) every free (projective) module over a right hereditary ring is a Rickart module, therefore $\mathbb{Z}$ considered as a $\mathbb{Z}$-module is $\Sigma$-Rickart, while $\mathbb{Z}$ is not a $\mathbb{Z}$-d-Rickart module. Hence, it can not be a {\it $\Sigma$-d-Rickart} module.
	\end{enumerate}
\end{example}
~\\
\begin{definition}
	An $A$-module $X$ is called $\Sigma$-$\operatorname{C2}$ module or said to have $\Sigma$-$\operatorname{C2}$ condition, if every direct sum of copies of $X$ is a $\operatorname{C2}$-module. Analogously, a module $X$ is said to be $\Sigma$-(quasi-)injective, if every direct sum of copies of $X$ is a (quasi-)injective module (see \cite{Albu}).
\end{definition}
~\\
\begin{proposition}\label{SDR1.2}
	A $\Sigma$-Rickart module with $\Sigma$-$\operatorname{C2} condition$ is {\it $\Sigma$-d-Rickart}.	
	\begin{proof}
		Let $A$ be a ring and $X$ be an $A$-module. To prove $X$ is a {\it $\Sigma$-d-Rickart} module, it is enough to show that $\operatorname{Im}(\psi)$ is a direct summand of $X^{(\Lambda)}$ for every $\psi \in \operatorname{End}_A( X^{(\Lambda)})$ where $\Lambda$ is a non-empty index set. For it, let $\psi \in \operatorname{End}_A( X^{(\Lambda)})$. Since $X$ is a $\Sigma$-Rickart module, $\operatorname{Ker}(\psi)$ is a direct summand of  $X^{(\Lambda)}$. So for some submodule $N$ of $X^{(\Lambda)}$, $X^{(\Lambda)}=\operatorname{Ker}(\psi)\oplus N$. Clearly, the restriction map $\psi _N$ is one-one. By assumption, $X$ is a $\Sigma$-$\operatorname{C2}$ module, so $X^{(\Lambda)}$ is a $\operatorname{C2}$ module. Therefore, $\psi(N)$ is a direct summand of $X^{(\Lambda)}$. Thus, $\operatorname{Im}(\psi)=\{ 0\}\oplus \psi(N)$ is a direct summand of $X^{(\Lambda)}$. Hence, $X$ is a {\it $\Sigma$-d-Rickart} module. 
	\end{proof}
\end{proposition}

The following example shows that the condition ``$X$ is $\Sigma$-$\operatorname{C2}$" in Proposition \ref{SDR1.2} is not superfluous.\\
\begin{example} 
	The $\mathbb{Z}$-module $\mathbb{Z}$ is a $\Sigma$-Rickart module but not a {\it $\Sigma$-d-Rickart} module and also not a $\Sigma$-$\operatorname{C2}$ module. In fact, $\mathbb{Z}$-module $\mathbb{Z}$ is not a d-Rickart module as well as not a $\operatorname{C2}$-module. 
\end{example}
~\\
\begin{corollary}\label{SDR1.3}
	If $A$ is a right Noetherian ring, then every injective $\Sigma$-Rickart $A$-module is {\it $\Sigma$-d-Rickart}.	
	\begin{proof}
		Let $A$ be a right Noetherian ring and $X$ be an injective $A$-module. It is well known that over the right Noetherian ring, every direct sum of copies of an injective module is injective. Therefore, $X^{(\Lambda)}$ is injective and so $X^{(\Lambda)}$ is a $\operatorname{C2}$ module. Hence, by Proposition \ref{SDR1.2} $X$ is a {\it $\Sigma$-d-Rickart} module.
	\end{proof}
\end{corollary}

\begin{corollary}
	Every $\Sigma$-quasi-injective, $\Sigma$-Rickart module is {\it $\Sigma$-d-Rickart}.
\end{corollary}
~\\
A module $X$ is said to be $\Sigma$-$\operatorname{C3}$ if every direct sum of copies of $X$ is a $\operatorname{C3}$-module \cite{C3}.\\
\begin{corollary}
	Every $\Sigma$-$\operatorname{C3}$, $\Sigma$-Rickart module is a {\it $\Sigma$-d-Rickart} module.
	\begin{proof}
		It is clear from (Corollary 2.7, \cite{C3}) that a right $A$-module $X$ is $\Sigma$-$\operatorname{C3}$ if and only if it is a $\Sigma$-$\operatorname{C2}$. Hence the result follows from Proposition \ref{SDR1.2}.
	\end{proof}
\end{corollary}	

\begin{definition}
	An $A$-module $X$ is called $\Sigma$-$\operatorname{D2}$  or said to have $\Sigma$-$\operatorname{D2}$ condition, if every direct sum of copies of $X$ is a $\operatorname{D2}$-module. Analogously, a module $X$ is said to be $\Sigma$-quasi-projective if every direct sum of copies of $X$ is quasi-projective \cite{Albu}.	
\end{definition}
~\\
In the following proposition, we show when a {\it $\Sigma$-d-Rickart} module is a $\Sigma$-Rickart module.\\
\begin{proposition}\label{SDR1.4}
	Every $\Sigma$-$\operatorname{D2}$, {\it $\Sigma$-d-Rickart} module are $\Sigma$-Rickart.
	\begin{proof}
		Let $\Lambda$ be any arbitrary index set and $\psi \in \operatorname{End}_A(X^{(\Lambda)})$. Since $X$ is a {\it $\Sigma$-d-Rickart}  module, $X^{(\Lambda)}/\operatorname{Ker}(\psi)\cong \operatorname{Im}(\psi)\leq ^{\oplus}X^{(\Lambda)}$. By hypothesis $X$ is a $\Sigma$-$\operatorname{D2}$ module, so $X^{(\Lambda)}$ has $\operatorname{D2}$-condition. Therefore, $\operatorname{Ker}(\psi)$ is direct summand of $X^{(\Lambda)}$. Hence, $X$ is a $\Sigma$-Rickart module.
	\end{proof}	
\end{proposition}

The following example shows that the condition ``$X$ is $\Sigma$-$\operatorname{D2}$" in Proposition \ref{SDR1.4} is not superfluous. \\

\begin{example}
	It is clear from Example \ref{SDR1.1} that the $\mathbb{Z}$-module $\mathbb{Z}_{p^{\infty}}^{(\Lambda)}$ is a d-Rickart module for every arbitrary index set $\Lambda$. Thus, $Z_{p^{\infty}}$ is a {\it $\Sigma$-d-Rickart} $\mathbb{Z}$-module, although it is not a $\Sigma$-$\operatorname{D2}$ $\mathbb{Z}$-module as it is not a $\operatorname{D2}$-module. While $Z_{p^{\infty}}$ is not a $\Sigma$-Rickart module because it is not a Rickart module.
\end{example} 
~\\
An $A$-module $X$ is called $\Sigma$-$\operatorname{D3}$  or said to have $\Sigma$-$\operatorname{D3}$ condition, if every direct sum of copies of $X$ is a $\operatorname{D3}$-module.\\

\begin{corollary}
	For an $A$-module $X$, the following statements are hold:
	\begin{enumerate}
		\item[(i)] Every projective {\it $\Sigma$-d-Rickart} module is $\Sigma$-Rickart.
		\item[(ii)] Every $\Sigma$-quasi-projective {\it $\Sigma$-d-Rickart} module is $\Sigma$-Rickart.
		\item[(iii)] Every $\Sigma$-$\operatorname{D3}$, {\it $\Sigma$-d-Rickart} module is {\it $\Sigma$-d-Rickart}.
	\end{enumerate}
	\begin{proof}
		It is easy to see that every projective and every $\Sigma$-quasi-projective modules are $\Sigma$-$\operatorname{D2}$ modules. Hence part $(i)$ and $(ii)$ are easily follows from Proposition \ref{SDR1.4}.\\
		For part $(iii)$, let $X$ be a $\Sigma$-$\operatorname{D3}$ module so by (Corollary 9, \cite{D3}) $X$ is a $\Sigma$-$\operatorname{D2}$ module. Hence, the proof is clear from Proposition \ref{SDR1.4}.
	\end{proof}
\end{corollary}

\begin{corollary}
	Let $A$ be a uniserial ring. Then every quasi-injective, {\it $\Sigma$-d-Rickart} module over $A$ is a $\Sigma$-Rickart module.
	\begin{proof}
		The proof follows from Proposition \ref{SDR1.4} and from the fact that every quasi-injective module over a uniserial ring is $\Sigma$-quasi-projective (see Theorem 5.1 \cite{Fuller}).
	\end{proof}	
\end{corollary}

In the following theorem, we characterize the semisimple Artinian ring in terms of {\it $\Sigma$-d-Rickart} modules and $\Sigma$-Rickart modules.\\

\begin{theorem}\label{SDR1.5}
	The following conditions are equivalent for a ring $A$:
	\begin{enumerate}
		\item[(i)] Every right $A$-module is $\Sigma$-Rickart;
		\item[(ii)] Every right $A$-module is {\it $\Sigma$-d-Rickart};
		\item[(iii)] $A$ is a right semisimple Artinian ring.
	\end{enumerate}
	\begin{proof}
		$(i)\Leftrightarrow (iii)$ Let $A$ be ring and $X$ be an $A$-module. It is clear from (Theorem 2.25, \cite{LRRR}) that a ring $A$ is right semisimple Artinian if and only if every right $A$-module is Rickart. Therefore, for any arbitrary index set $\Lambda$, $X^{(\Lambda)}$ is a Rickart module if and only if $A$ is a right semisimple Artinian ring. Hence, $A$ is a right semisimple Artinian ring if and only if every right $A$-module is $\Sigma$-Rickart.\\
		$(ii) \Leftrightarrow (iii)$ Recall from (Theorem 2.24, \cite{LRRD}) that a ring $A$ is right semisimple Artinian if and only if every right $A$-module is d-Rickart. Therefore, for any arbitrary index set $\Lambda$, $X^{(\Lambda)}$ is a d-Rickart module if and only if $A$ is a right semisimple Artinian ring. Hence, $A$ is a right semisimple Artinian ring if and only if every right $A$-module is {\it $\Sigma$-d-Rickart}.
	\end{proof}
\end{theorem}


\section{The Endomorphism ring of a {\it $\Sigma$-d-Rickart} module}

In this section, we study the endomorphism ring of a {\it $\Sigma$-d-Rickart} module and characterize the semi-hereditary rings, hereditary rings and regular rings in terms of it. \\

\begin{proposition}\label{S3.2}
	If $X$ is a {\it $\Sigma$-d-Rickart} $A$-module, then the endomorphism ring  $S=\operatorname{End}_A(X)$ is (left) semi-hereditary. 
	\begin{proof}
		Let $X$ be a {\it $\Sigma$-d-Rickart} module. Then, $X^{(n)}$ is a d-Rickart module for every $n\in \mathbb{N}$. From  (Proposition  3.1, \cite{LRRD}), $\operatorname{End}_A(X^{(n)})\cong Mat_n(S)$ is a left Rickart ring. Hence, $S$ is a left semi-hereditary from  (Proposition, \cite{LW}).
	\end{proof}
\end{proposition}
~\\
\begin{proposition}\label{S3.3}
	If $X$ is a right {\it $\Sigma$-d-Rickart} module and $S=\operatorname{End}_A(X)$, then ${_SX}$ is a FP-injective $S$-module.
	\begin{proof}
		Let $X$ be a {\it $\Sigma$-d-Rickart} module and $\phi : X^{(n_1)}\rightarrow X^{(n_2)}$ be any homomorphism. Then, $\operatorname{Im}(\phi)\leq^{\oplus}X^{(n_2)}$ which implies that $\operatorname{Coker}(\phi)=X^{(n_2)}/ \operatorname{Im}(\phi)$ is isomorphic to a direct summand of $X^{(n_2)}$. Thus, $\operatorname{Coker}(\phi)$ is $X$-cogenerated. Hence, from Lemma \ref{FP1.4} ${_SX}$ is FP-injective $S$-module.
	\end{proof}
\end{proposition}
~\\
\begin{proposition}\label{S3.4}
	If $X$ is a finitely generated {\it $\Sigma$-d-Rickart} module with endomorphism ring $S=\operatorname{End}_A(X)$, then $S$ is a left hereditary ring.
	\begin{proof}
		Since $X$ is a finitely generated module, $\operatorname{End}_A(X^{(\Lambda)})\cong \operatorname{End}_{S}({S}^{(\Lambda)})$ for any non-empty index set $\Lambda$. By hypothesis, $X$ is a {\it $\Sigma$-d-Rickart} module, therefore $X^{(\Lambda)}$ is d-Rickart. From (Proposition 3.1, \cite{LRRD}) $\operatorname{End}_A(X^{(\Lambda)})$ is a left Rickart ring. Thus, $\operatorname{End}_{S}({S}^{(\Lambda)})=\operatorname{CFM}_{\Lambda}(S)$ is a left Rickart ring. Hence, $S$ is a left hereditary ring from (Proposition 3.20, \cite{LRRS}).	
	\end{proof}
\end{proposition}

The following example shows that the finitely generated condition in Proposition \ref{S3.4} is not superfluous. \\

\begin{example}\label{S3.1}
	From (Theorem 2.29, \cite{LRRD}), it is clear that for a right hereditary ring $A$, $E(X)$ is a d-Rickart $A$-module for any right $A$-module $X$. Therefore, $\mathbb{Q}^{(\Lambda)}$ is a {\it $\Sigma$-d-Rickart} $\mathbb{Z}$-module, where $\Lambda$ is an arbitrary index set. Clearly, $\mathbb{Q}^{(\Lambda)}$ is not finitely generated $\mathbb{Z}$-module.
	Since $\operatorname{End}_{\mathbb{Z}}(\mathbb{Q}^{(\Lambda)}) \cong \operatorname{CFM}_{\Lambda}(\mathbb{Q})$ contains $\Pi _{\lambda \in \Lambda} A_{\lambda}$ where $A_{\lambda} =\mathbb{Q}$ for each $\lambda \in \Lambda$, therefore by Lemma \ref{LI}, $\operatorname{End}_{\mathbb{Z}}(\mathbb{Q}^{(\Lambda)})$ is not a right hereditary ring.
\end{example}
~\\
The following proposition illustrates when the endomorphism ring of a {\it $\Sigma$-d-Rickart} module is a von Neumann regular.\\

\begin{proposition}\label{S3.7}
	If $X$ is a projective {\it $\Sigma$-d-Rickart} module, then the endomorphism ring  $S=\operatorname{End}_A(X)$ is a von Neumann regular.
	\begin{proof}
		Let $X$ be a {\it $\Sigma$-d-Rickart} module and $f \in \operatorname{End}_A(X^{(\Lambda)})$ be an endomorphism. Then, $X^{(\Lambda)}$ is a d-Rickart module (where $\Lambda$ be an index set) and $X/\operatorname{Ker}(f)\cong \operatorname{Im}(f)\leq ^{\oplus} X^{(\Lambda)}$. Since $X$ is a projective module, $\operatorname{Ker}(f)\leq ^{\oplus} X^{(\Lambda)}$. Hence, $\operatorname{End}_A(X^{(\Lambda)})\cong \operatorname{CFM}_{\Lambda} (S)$ is a von Neumann regular which implies that $S$ is a von Neumann regular ring. 
	\end{proof}
\end{proposition}

\begin{corollary}
	Every projective {\it $\Sigma$-d-Rickart} module is endoregular.
\end{corollary}

\section*{Acknowledgment}
The first author gratefully acknowledges UGC, INDIA for providing JRF-SRF fellowship to carry out this research work.


\end{document}